\documentclass[11pt]{article}
\usepackage{amsmath}
\usepackage{amsthm}
\usepackage{amsfonts}
\usepackage{amssymb}
\usepackage{xcolor}
\usepackage{mathtools}
\usepackage[latin1]{inputenc}
\usepackage{hyperref}

\allowdisplaybreaks

\newtheorem{theorem}{Theorem}[section]
\newtheorem{proposition}[theorem]{Proposition}
\newtheorem{lemma}[theorem]{Lemma}
\newtheorem{corollary}[theorem]{Corollary}
\newtheorem{definition}[theorem]{Definition}

\newtheorem{example}[theorem]{Example}

\title{\bf Controllability of time-varying lumped semilinear
systems}

\author{
Hern\'an R. Henr\'{\i}quez\thanks{Departamento de Matem\'atica y Ciencia de la Computaci\'on, Universidad de Santiago de Chile, Casilla 307, Correo 2, Santiago, Chile. hernan.henriquez@usach.cl} and Matthieu F. Pinaud\thanks{
Departamento de Matem\'atica y Ciencia de la Computaci\'on, Universidad de Santiago de Chile, Casilla 307, Correo 2, Santiago, Chile. matthieu.pinaud@usach.cl}\\ \\ 
{\it Dedicated to the memory of Hern\'an R. Henr\'{\i}quez.}
}

\date{ }

\begin{document}
\maketitle
\begin{abstract}
In this work we are concerned with the controllability of time-varying lumped control systems
governed by a semilinear differential equation. We consider the semilinear system as a perturbation of a linear system. Assuming the underlying linear system is controllable, and the nonlinear forcing function satisfies a boundedness condition which is adapted to the underlying linear system, we show that the semilinear system is also approximately controllable
\end{abstract}
\bigskip
\noindent
\textbf{MSC 2020 subject classification}: 93B05, 93C10 (Primary); 93C15, 93C05 (Secondary).\\
\textbf{Keywords}: Controllability of systems; systems governed by ordinary differential equations; nonlinear systems.

\maketitle

\section{Introduction} \label{Intr}
Nowadays it is well known that the controllability of systems is a fundamental concept in the theory
of systems with many important applications. For this reason, over the past decades the issue of
controllability of finite dimensional systems has been studied by many authors using diverse techniques. In particular, the controllability of nonlinear systems has attracted the attention of numerous researchers since the 1960 and currently has an extensive literature. Many of the advances in the theory have already been incorporated into books. In relation to the objectives of this work, we will limit ourselves to mentioning \cite{1, 6, 8, 9, 12, 20, 26}, where the controllability of the systems described in form
\begin{equation*}
x^{\prime}(t)  =  f(t,x(t),u(t)),\quad x(t)\in \mathbb{R}^n, u(t)\in \mathbb{R}^m, 0\leq t\leq T, \tag{1.1}
\end{equation*}
is studied from a geometric point of view, for which a recurring hypothesis is that $f$ is a smooth function, usually of class $C^\infty$, or analytic. In \cite{6} geometric tools are used to analyze the controllability of nonlinear mechanical systems, and in \cite{8} the control sets are introduced and their properties are studied. In \cite{12} the controllability of the nonlinear system is discussed both by comparison with the linearized system and using geometric techniques. Monograph \cite{30} contains an overview of the main topics of mathematical control theory, especially controllability, stabilization and optimal control, for both linear and nonlinear systems, and for systems with states in spaces of finite and infinite dimension. Additionally, Chukwu \cite{7} establishes (1.1) the null controllability of the system $S$ under conditions in the Jacobian matrix of $f$, some stability properties and the controllability of a linear
system related to (1.1).\\

\noindent
The controllability of the semilinear system
\begin{equation*}
x^{\prime}(t)  =  A(t) x(t)+f(t,x(t))  + B(t) u(t),\quad 0 \leq  t \leq T, \tag{1.2}
\end{equation*}
with states $x(t) \in \mathbb{R}^n$ and controls $u(t) \in \mathbb{R}^m$ has been studied by several authors in recent years. In (1.2), $A(t)$, $B(t)$ are $n \times n$ and $n \times m$ matrices, respectively, $f\colon [0,T] \times \mathbb{R}^n \to \mathbb{R}^n$ is a function whose
properties will be specified later, and the control function $u \in L^p([0,T],\mathbb{R}^m)$ for $1 \leq p \leq \infty$. The general approach consists in to assume that the linear system
\begin{equation*}
x^{\prime}(t)  =  A(t) x(t)  + B(t) u(t),\quad 0 \leq  t \leq T, \tag{1.3}
\end{equation*}
is controllable and to determine conditions on f so that the system (1.2) is also controllable. In \cite{13} the authors show that if the system (1.3) is controllable and
\[ \lim_{\lVert (x,u)\lVert\to \infty} \frac{f(t,x,u)}{\lVert (x,u)\lVert}\to 0, \]
uniformly for $t \in [0,T]$, then system (1.2) is also controllable. Using the Schauder fixed point theorem it was shown in \cite{15} that if
\[
\lVert f(t,x(t),u(t)) \rVert \leq \sum_{i=1}^{q} \alpha_i(t)\phi_i(x,u), \quad 0\leq t\leq T
\]
where functions $\alpha_i(\cdot)$ and $\phi_i(\cdot)$ satisfy certain technical conditions, then the controllability of the system (1.3) implies the controllability of system (1.2). 
It is important to note here that several recent works study the controllability of fractional systems. In \cite{4,5}, their authors study the controllability of a semilinear fractional system of order $1<\alpha<2$. Expressing the solution in terms of the Mittag-Leffler matrix function, and using the Schauder fixed point theorem, their authors establish the controllability of the semilinear system. It is also worth to mention that several authors have studied the controllability of distributed semilinear systems, that is systems modeled in infinite dimensional spaces. The basic model considered in these works is
\begin{equation*}
x^{\prime}(t)  =  A  x(t) +f(t,x(t)) + Bu(t), \;\; 0 \leq  t \leq T, \tag{1.4}
\end{equation*}
where $x(t) \in X$, $u(t) \in U$, $X$, $U$ are Banach spaces, the operator $A$ is the infinitesimal generator of a strongly continuous semigroup on $X$ and $B \colon U \to X$ is a bounded linear operator. Zhou \cite{31} studies the approximate controllability of semilinear distributed systems of type (1.4) modeled in Hilbert spaces. He establishes the existence of mild solutions and sufficient conditions to guarantee the approximate controllability of system (1.4). In \cite{27} the controllability of system (1.4) is considered using fixed point techniques, and in \cite{14,25} the approximate controllability of systems of type (1.4) described in Hilbert spaces is studied. Under the assumption that the linear system is approximately controllable, Naito \cite{25} establishes the approximate controllability of (1.4) in terms of a geometric property of involved operators, and Do \cite{14} shows the approximate controllability of (1.4) when
\[ \lim_{\lVert x \lVert \to \infty} \frac{f(t,x)}{\lVert x \lVert}=\varepsilon,\]
for $\varepsilon > 0$ sufficiently small. In \cite{28} the controllability of a non-autonomous distributed system (in the sense that operator $A$ is time-dependent) is examined. Assuming that the linear system is controllable and that certain technical conditions are satisfied, including a boundedness condition for $f$, the author establishes the controllability of the system (1.4). This type of results has also been developed for more general systems. Using the Schauder fixed point theorem, in \cite{2} the authors study the controllability of a distributed integro-differential system. The paper \cite{3} presents a survey of the controllability properties for various distributed systems. Wang \cite{29} discuss the approximate controllability of a control system governed by a semilinear integro-differential equation. Under the condition that the corresponding linear system is approximately controllable and the function $f$ is bounded the author gets to establish the approximate controllability of the semilinear system. Using the Hilbert Uniqueness Method, the exact controllability of the wave equation is studied in \cite{32,33}. The author establishes the exact controllability of the semilinear equation when the nonlinearity $f$ is asymptotically linear or satisfies a given condition of growth at infinity. In \cite{17} the authors prove the approximate controllability for the semilinear heat equation assuming that the nonlinearity $f$ is globally Lipschitz continuous. Fern\'andez and Zuazua \cite{18} study the approximate controllability of the semilinear heat equation when the perturbation $f$ also depends on the temperature gradient. The approach in this case is based on optimization techniques. In \cite{10,11} the authors examine then exact boundary controllability of a semilinear heat equation and wave equation both in one dimension. The purpose of this work is to study the controllability of system (1.2).\\

\noindent
To obtain our results, in this work we assume that the underlying linear system is controllable in a sense that will be explained in Section 3, and that $f$ satisfies some boundedness condition. It is important to mention that we consider the system (1.2) as a perturbation of the linear system (1.3). In this work, this situation is reflected in two aspects. First, the existence of solutions of (1.2) depends on a boundedness property of $f$, which is not maintained if we redefine $f$ in the form $-A(t)x(t)+f(t,x)$. Something similar occurs with the controllability property. The controllability of system (1.2) depends on the controllability of system (1.3), a property that is not maintained if we redefine $f$.\\

\noindent
We next introduce some notations that will be widely used throughout the text. We denote by
$C([0,a],\mathbb{R}^n)$ the space of continuous functions from $[0,a]$ into $\mathbb{R}^n$ endowed with the norm of the uniform convergence $\lVert\cdot\rVert_\infty$; for a matrix $M$ or a linear operator $M$, we denote by $\operatorname{Im}(M)$ the range space of $M$, and by $\sigma(M)$ the set of eigenvalues of $M$. In this text we will consider the space $\mathbb{R}^k$ provided with the Euclidean norm which will be denoted by $\lVert\cdot\rVert$ independent of the dimension of space.
Additionally, the norm of an $l \times k$ matrix $M$, also denoted by $\lVert\cdot\rVert$, will be the norm of $M$ considered as a linear transformation from $\mathbb{R}^k$ into $\mathbb{R}^l$, that is 
\[\lVert M\rVert=\sup\{\lVert Mx\rVert:x\in\mathbb{R}^k,\ \lVert x\rVert\leq 1\}.\]
This work is organized in two parts. Section 2 contains preliminary aspects on the controllability and a geometric property of the trajectories of linear systems, and the results of the existence of solutions for the semilinear system. In Section 3 we study the controllability of the semilinear system (1.2). Expressing it in a very superficial way, the main result of this work is Theorem 3.3, which can be abbreviated as follows:
\begin{theorem}
Assume that the system (1.3) is uniformly controllable on $[0,T]$, the function $\lVert A(\cdot)\rVert$ is bounded on $[0,T]$, and technical conditions $(C1)$-$(C5)$ are satisfied. Then system (1.2) is approximately controllable on $[0,T]$.
\end{theorem}

\section{Preliminaries} \label{Prel}
In the first part of this section we review briefly some basic aspects relative to the controllability of systems described by ordinary differential equations, and which are needed to establish our results. We refer the reader to \cite{21} for the controllability concepts.\\

\noindent
Throughout this work we will assume that $A(t)$ and $B(t)$ are matrices of order $n \times n$ and $n \times m$, respectively, and that the matrix functions $A\colon [0,T]\to\mathbb{R}^{n\times n}$,
$t\mapsto A(t)$, and $B\colon [0,T]\to\mathbb{R}^{n\times m}$, $t\mapsto B(t)$, are integrable. Let $h\colon [0,T]\to\mathbb{R}^n$ be an integrable function. Then for every $x_0\in\mathbb{R}^n$ there exists a unique absolutely continuous function $x(\cdot)$ which is solution of the system
\begin{align*}
    x'(t)&=A(t)x(t)+h(t),\quad t\in[0,T], \tag{2.1}\\
    x(0)&=x_0. \tag{2.2}
\end{align*}
Using the terminology of \cite{23}, we denote $\Phi(t,t_0)$ the fundamental solution matrix (also called transition matrix in \cite{21}) corresponding to the homogeneous equation
\[
x'(t)=A(t)x(t),\quad t\geq t_0.
\]
Using the fundamental solution matrix the solution of problem (2.1)-(2.2) is given by
\[
x(t)=\Phi(t,0)x_0+\int_0^t\Phi(t,s)h(s)\,ds.
\]
In this section we are concerned with the linear time-varying control system (1.3) where $u(t)\in\mathbb{R}^m$ is the control, and we take as admissible controls the functions $u\in L^p([0,T],\mathbb{R}^m)$ for a fixed $1\leq p\leq\infty$.
Let $q$ be the conjugate exponent of $p$. Henceforth the text we will assume that $B(\cdot)\in L^r([0,T],\mathbb{R}^{n\times m})$, where $r\geq\max\{p,q\}$. We denote by $x(t;t_0,x_0,u)$ for $t\geq t_0$ the solution of (1.3) with initial condition $x_0$ at $t=t_0$. When $t_0=0$, we write $x(t;x_0,u)$ instead of $x(t;0,x_0,u)$. From \cite{21}, we consider the
following notions of controllability.

\begin{definition}
The system (1.3) is said to be controllable in $[t_0,t_1]$, for $0\leq t_0\leq t_1\leq T$, if for every
$x_0,x_1\in\mathbb{R}^n$ there exists a control function $u\in L^p([t_0,t_1],\mathbb{R}^m)$ such that $x(t_1;t_0,x_0,u)=x_1$.
\end{definition}

\begin{definition}
The system (1.3) is said to be uniformly controllable in $[0,T]$ if it is controllable in
any interval $[t_0,t_1]$ for $0\leq t_0<t_1\leq T$.
\end{definition}

\noindent
The following result is well known (\cite[Theorem 1.3.1]{21}).

\begin{theorem}
The system (1.3) is controllable in $[t_0,t_1]$ if and only if the Gramian controllability
matrix
\[
W(t_1,t_0)=\int_{t_0}^{t_1}\Phi(t_1,s)B(s)B^*(s)\Phi^*(t_1,s)\,ds
\]
is positive definite.
\end{theorem}

\noindent
We next fix $t_0=0$ and $t_1=\tau$. We abbreviate the notation by writing $\Phi(t,0)=\Phi(t)$. We define
the matrix
\[
M(\tau)=\int_0^\tau\Phi(s)^{-1}B(s)B^*(s)\Phi^*(s)^{-1}\,ds. \tag{2.3}
\]
Since the matrix $\Phi(t_1,t_0)$ is invertible for all $0\leq t_0\leq t_1$, for every $0\leq s\leq\tau$ we have
\[
\Phi(\tau)=\Phi(\tau,0)=\Phi(\tau,s)\Phi(s,0)=\Phi(\tau,s)\Phi(s)
\]
which implies that $\Phi(\tau,s)=\Phi(\tau)\Phi(s)^{-1}$. Hence
\[
W(\tau,0)
=\Phi(\tau)\int_0^\tau\Phi(s)^{-1}B(s)B^*(s)\Phi^*(s)^{-1}\,ds\,\Phi(\tau)^*
=\Phi(\tau)M(\tau)\Phi(\tau)^*.
\]
It follows from the preceding expression and Theorem 2.3 that the system (1.3) is controllable in $[0,\tau]$ if and only if the matrix $M(\tau)$ is positive definite. Assume the system (1.3) is controllable on $[0,\tau]$ for all $\tau>0$. Let $x_0,x_1\in\mathbb{R}^n$ be fixed. We
consider the set consisting of trajectories $x(t;x_0,u)$ such that $x(\tau;x_0,u)=x_1$, and we denote by $U(\tau;x_0,x_1)$ the set formed by the corresponding admissible controls $u\in L^p([0,\tau],\mathbb{R}^m)$. We also denote by $B_\tau:L^p([0,\tau],\mathbb{R}^m)\to\mathbb{R}^n$ the bounded linear operator given by
\[
B_\tau(u)=\int_0^\tau B(s)u(s)\,ds.
\]
It is clear that $\operatorname{Im}(B_\tau)$ is a closed vector subspace of $\mathbb{R}^n$. We can establish the following property of trajectories.

\begin{proposition}
Assume that the matrix functions $A:[0,T]\to\mathbb{R}^{n\times n}$, $t\mapsto A(t)$, and $B:
[0,T]\to\mathbb{R}^{n\times m}$, $t\mapsto B(t)$, are integrable. If $d(x_1-x_0,\operatorname{Im}(B_{\tau_0}))>d_0\geq 0$, for some $\tau_0>0$,
then 
\[ \lim_{\tau\to 0^+} \sup\{\lVert x(t;x_0,u)\rVert:0\leq t\leq\tau,u\in U(\tau;x_0,x_1)\}=\infty.\]
\end{proposition}

\begin{proof}
Let $x(t)=x(t;x_0,u)$ for $u\in U(\tau;x_0,x_1)$. It follows from (1.3) that
\[
x(\tau)-x_0=x_1-x_0=\int_0^\tau A(s)x(s)\,ds+\int_0^\tau B(s)u(s)\,ds.
\]
Since $\operatorname{Im}(B_\tau)\subseteq\operatorname{Im}(B_{\tau_0})$ for $0<\tau\leq\tau_0$, we obtain that
\[
d_0<\lVert x_1-x_0-B_\tau(u)\rVert
=\left\lVert\int_0^\tau A(s)x(s)\,ds\right\rVert
\leq\int_0^\tau\lVert A(s)\rVert\,ds\sup_{0\leq s\leq\tau}\lVert x(s)\rVert.
\]
Using now that $\lVert A(\cdot)\rVert$ is integrable, we have that $\int_0^\tau\lVert A(s)\rVert\,ds\to 0$ as $\tau\to 0$, and combining with the
above inequality follows that $\sup_{0\leq s\leq\tau}\lVert x(s)\rVert\to\infty$ as $\tau\to 0$.
\end{proof}
\noindent
This property has a particularly simple expression for invariant linear systems
\[
x'(t)=Ax(t)+Bu(t),\quad t\in[0,T], \tag{2.4}
\]
where $A,B$ are $n\times n$ and $n\times m$ constant matrices, respectively.

\begin{corollary}
For the invariant system (2.4), if $d(x_1-x_0,\operatorname{Im}(B))>d_0\geq 0$, then 
\[ \lim_{\tau\to 0^+} \sup\{\lVert x(t;x_0,u)\rVert:0\leq t\leq\tau,u\in U(\tau;x_0,x_1)\}=\infty.\]
\end{corollary}

Not to doubt that this property is difficult to understand when we compare it with our intuitive idea
of controllability applied to concrete systems. The reason is due to the following well known property, which allows us to justify our intuition through the idea of approximate control. For completeness we first recall the following property of matrices. For a proof the reader can see \cite{22}.

\begin{lemma}
Let $M$ be a $n\times n$ positive definite matrix. Then $\lVert M\rVert=\sup\sigma(M)$.
\end{lemma}
\noindent
Combining Theorem 2.3 and Lemma 2.6 the next result follows. This result is well known, however,
for completeness of the text we will include a brief idea of its demonstration.

\begin{corollary}
The system (1.3) is controllable in $[0,\tau]$ if and only if
\[
\lim_{\varepsilon\to 0}\varepsilon(\varepsilon I+M(\tau))^{-1}=0.
\]
\end{corollary}

\begin{proof}
It is easy to see that $\lambda\in\sigma(M(\tau))$ if and only if $\frac{1}{\varepsilon+\lambda}\in
\sigma((\varepsilon I+M(\tau))^{-1})$. Assume that system
(1.3) is controllable in $[0,\tau]$. This implies that $M(\tau)$ is positive definite. Consequently, there exists $\lambda_0>0$ such that $\lambda\geq\lambda_0$ for all $\lambda\in\sigma(M(\tau))$. Therefore, from Lemma 2.6 we obtain that
\[
\left\lVert\varepsilon(\varepsilon I+M(\tau))^{-1}\right\rVert
\leq\frac{\varepsilon}{\varepsilon+\lambda_0}\to 0,\quad \varepsilon\to 0.
\]
Reciprocally, using again Lemma 2.6, for every $\lambda\in\sigma(M(\tau))$,
\[
\frac{\varepsilon}{\varepsilon+\lambda}\leq\lVert\varepsilon(\varepsilon I+M(\tau))^{-1}\rVert\to 0,\quad \varepsilon\to 0.
\]
This implies that there exists $\varepsilon_0>0$ such that $\lambda\geq\varepsilon_0$, which in turn implies that the matrix $M(\tau)$ is positive definite.
\end{proof}
\noindent
We also need to study the controllability of translated systems. Let $0\leq\omega<T$. We begin by
observing that the fundamental matrix $\Phi_\omega(t,s)$ corresponding to equation
\[ x'(t)=A(t+\omega)x(t),\quad t\in[0,T-\omega],\]
is given by
\[ \Phi_\omega(t,s)=\Phi(t+\omega,s+\omega).\]
Furthermore, related with the controllability of the translated system
\[ x'(t)=A(t+\omega)x(t)+B(t+\omega)u(t),\quad t\in[0,T-\omega], \tag{2.5} \]
we have the following property.

\begin{proposition}
Assume that the system (1.3) is uniformly controllable in $[0,T]$. Then the system
(2.5) is controllable in $[0,T-\omega]$ for all $0\leq\omega<T$.
\end{proposition}

Let $M_\omega(\tau)$ with $0<\tau\leq T-\omega$ be the matrix defined in (2.3) for the system (2.5). This means that
\[ M_\omega(\tau)=\int_\omega^{\omega+\tau}\Phi(s,\omega)^{-1}B(s)B^*(s)\Phi^*(s,\omega)^{-1}\,ds. \tag{2.6} \]
In what follows we abbreviate the notation by writing
\[ M_\omega=M_\omega(T-\omega)=\int_\omega^T\Phi(s,\omega)^{-1}B(s)B^*(s)\Phi^*(s,\omega)^{-1}\,ds. \tag{2.7} \]
We denote by $\beta_\omega(\tau)$ the lowest eigenvalue of $M_\omega(\tau)$. As a consequence, for $\varepsilon>0$, we have that
\[ \lVert(\varepsilon I+M_\omega)^{-1}\rVert\leq\frac{1}{\varepsilon+\beta_\omega(T-\omega)}. \]
We now consider the semilinear equation
\[ x'(t)=A(t)x(t)+f(t,x(t))+h(t),\quad t\in[0,T], \tag{2.8}, \]
where $h:[0,T]\to\mathbb{R}^n$ is an integrable function and $f:[0,T]\times\mathbb{R}^n\to\mathbb{R}^n$ satisfies the following Carath\'eodory conditions:
\begin{itemize}
\item[\textbf{(C1)}] The function $f(\cdot,x):[0,T]\to\mathbb{R}^n$ is integrable for each $x\in\mathbb{R}^n$.
\item[\textbf{(C2)}] The function $f(t,\cdot):\mathbb{R}^n\to\mathbb{R}^n$ is continuous for each $t\in[0,T]$.
\end{itemize}
\noindent
In the literature there are numerous results on the existence and uniqueness of solutions for equation
(2.8) with initial condition (2.2). In what follows, we will present a result that is appropriate for
developing our future goals. For this reason, we introduce the following conditions.

\begin{itemize}
\item[\textbf{(C3)}] The function $f$ is locally Lipschitz continuous. That is, for every $x_0\in\mathbb{R}^n$ and $r>0$ there is
an integrable function $L_r:[0,T]\to[0,\infty)$ such that
\[ \lVert f(t,x_1)-f(t,x_2)\rVert\leq L_r(t)\lVert x_1-x_2\rVert \]
for all $x_i\in\mathbb{R}^n$ such that $\lVert x_i-x_0\rVert\leq r$ for $i=1,2$.
\item[\textbf{(C4)}] There exists a nondecreasing continuous function $\Omega:[0,\infty)\to[0,\infty)$ such that the function
$\phi:[0,\infty)\to[0,\infty)$ given by
\[
\phi(s)=\int_0^s\frac{1}{\Omega(\xi)}\,d\xi,\quad s\geq 0,
\]
is surjective, and there is an integrable function $m_f:[0,T]\to[0,\infty)$ such that
\[
\lVert f(t,x)\rVert\leq m_f(t)\Omega(\lVert x\rVert),
\]
for all $t\in[0,T]$ and $x\in\mathbb{R}^n$. 
\end{itemize}

\noindent
In the rest of this text, we will represent by $\mathcal{C}$ the class of functions $\Omega(\cdot)$ that satisfy the conditions
involved in (C4).

\begin{example}
The function $\Omega(s)=(s+e)\ln(s+e)$ satisfies the conditions involved in (C4). In fact,
$\Omega$ is a nondecreasing continuous function, and
\[
\phi(s)=\int_0^s\frac{1}{(\xi+e)\ln(\xi+e)}\,d\xi
=\ln\bigl(\ln(s+e)\bigr),\quad s\geq 0,
\]
is well defined and surjective from $[0,\infty)$ into $[0,\infty)$.
\end{example}

\noindent
Our result of existence of solutions for equation (2.8) is based on the following property of fixed
point, known as the Leray-Schauder's Alternative theorem (\cite[Theorem 6.5.4]{19}), and on a generalization due to Bihari of the Gronwall-Bellman lemma (see \cite[Theorem 4]{16}). However, since the proof in \cite[Theorem 4]{16} contains some confusing aspects, we will include a proof for completeness. In the rest of this work we denote by $N_\Phi$ a constant such that $\lVert\Phi(t,s)\rVert\leq N_\Phi$ for all $s,t\in[0,T]$.

\begin{theorem}
Let $D$ be a closed convex subset of a Banach space $(X,\lVert\cdot\rVert_X)$ and assume that
$0\in D$. Let $\Gamma:D\to D$ be a completely continuous map, then the set
$\{x\in D:x=\lambda\Gamma(x)$, for some $0<\lambda<1\}$ is unbounded or the map $\Gamma$ has a fixed point in $D$.
\end{theorem}

\begin{lemma}
Assume that $\Omega(\cdot)\in\mathcal{C}$. For every constant $C>0$ there exists a constant $\Psi(C)$ having the following property: let $\alpha(\cdot)\in C([0,T])$ be a positive function that satisfies
\[ \alpha(t)\leq C+N_\Phi\int_0^t m_f(s)\Omega(\alpha(s))\,ds,\quad 0\leq t\leq T,\]
then $\alpha(t)\leq\Psi(C)$ for all $t\in[0,T]$.
\end{lemma}

\begin{proof}
Let $\eta(t)=N_\Phi\int_0^t m_f(s)\Omega(\alpha(s))\,ds$. Then
\[ \eta'(t)=N_\Phi m_f(t)\Omega(\alpha(t))\leq N_\Phi m_f(t)\Omega(C+\eta(t)),\quad\text{a.e.}
\]
It follows that
\[ \frac{\eta'(t)}{\Omega(C+\eta(t))}\leq N_\Phi m_f(t), \]
which implies that
\[ \int_C^{C+\eta(t)}\frac{d\xi}{\Omega(\xi)} =\phi(C+\eta(t))-\phi(C) \leq N_\Phi\int_0^T m_f(s)\,ds.\]
Hence, we obtain that
\[ \phi(C+\eta(t))\leq\phi(C)+N_\Phi\int_0^T m_f(s)\,ds. \]
Since the function $\phi$ has a continuous and nondecreasing inverse $\phi^{-1}:[0,\infty)\to[0,\infty)$, the above inequality allows us to state that
\[ \alpha(t)\leq\Psi(C) =\phi^{-1}\left(\phi(C)+N_\Phi\int_0^T m_f(s)\,ds\right) \]
which completes the proof.
\end{proof}

\noindent
Throughout this paper, $B_r(x,X)$ will denote the closed ball with center at $x$ and radius $r>0$ in
a Banach space $X$.

\begin{theorem}
Assume conditions (C1)-(C4) are fulfilled, and $h\in L^1([0,T],\mathbb{R}^n)$. Then problem
(2.8)-(2.2) has a unique solution.
\end{theorem}

\begin{proof}
We define the map $\Gamma:C([0,T],\mathbb{R}^n)\to C([0,T],\mathbb{R}^n)$ by
\[
\Gamma x(t)=\Phi(t,0)x_0+\int_0^t\Phi(t,s)f(s,x(s))\,ds
+\int_0^t\Phi(t,s)h(s)\,ds,\quad 0\leq t\leq T.
\]
Applying the dominated convergence theorem of Lebesgue we can affirm that $\Gamma$ is continuous.
We next prove that $\Gamma$ is completely continuous. Initially we observe that
\[
\lVert f(s,x(s))\rVert\leq m_f(s)\Omega(\lVert x\rVert_\infty),
\]
for $x\in C([0,T],\mathbb{R}^n)$ and $s\in[0,T]$. Let $r>0$ and $B_r=B_r(0,C([0,T],\mathbb{R}^n))$. Using the above estimate we can show that the set $\Gamma(B_r)$ is bounded and equicontinuous. The Ascoli-Arzel\`a theorem allows us to conclude that $\Gamma(B_r)$ is relatively compact in $C([0,T],\mathbb{R}^n)$. Thus, $\Gamma$ is completely continuous. To finish the proof, we establish a priori estimates for the solution of the integral equation $\lambda\Gamma x=x$. Let $x_\lambda$ such that $\lambda\Gamma x_\lambda=x_\lambda$. For $t\geq 0$, we have
\[
\lVert x_\lambda(t)\rVert
\leq N_\Phi\lVert x_0\rVert+N_\Phi\int_0^T\lVert h(s)\rVert\,ds
+N_\Phi\int_0^t m_f(s)\Omega(\lVert x_\lambda(s)\rVert)\,ds.
\]
Applying Lemma 2.11, for a fixed function $h$, we obtain that the set
\[ \{x_\lambda:0<\lambda<1,\lambda x=\lambda\Gamma x\}\] 
is bounded in $C([0,T],\mathbb{R}^n)$. Applying now Theorem 2.10 we obtain the existence of a fixed point for $\Gamma$, and as a consequence, the existence of a solution for problem (2.8)-(2.2).
Finally, the uniqueness of the solution is a direct consequence of condition (C3).
\end{proof}

\section{Approximate controllability of nonlinear systems} \label{Contr}
This section is devoted to study the controllability of system (1.2). To establish our controllability
result for the semilinear system, we use a technique developed in \cite[Theorem 1]{29}. In this result
the controllability of the semilinear system is obtained assuming that the underlying linear system is
controllable and that $f$ is bounded. In this section we shall show that the nonlinear system is approximately controllable when the linear system is uniformly controllable and $f$ satisfies a boundedness condition which is related with the linear system. Our results are based on the existence of certain bounded trajectories of the linear system which will be defined below. 

\noindent
In this section we keep all the notations introduced in the Section 2. We assume that $A(\cdot)$, $B(\cdot)$ and $f$ satisfy the conditions considered in Section 2. Let $x(\cdot;x_0,f,u)$ be the solution of system (1.2)-(2.2) for a control function $u\in L^p([0,\tau],\mathbb{R}^m)$. We denote 
\[ R(\tau;x_0,f)=\{x(\tau;x_0,f,u):u\in L^p([0,\tau],\mathbb{R}^m)\}\] 
the set formed by the reachable states of system (1.2)-(2.2) in the interval $[0,\tau]$.

\begin{definition}
The system (1.2) is said to be approximately controllable in $[0,T]$ if the set $R(T,x_0,f)$
is dense in $\mathbb{R}^n$ for all $x_0\in\mathbb{R}^n$.
\end{definition}

\noindent
Our next objective is to study in greater detail some trajectories of the linear system (1.3). Let
$x_0,z\in\mathbb{R}^n$. Using the Corollary 2.7 and the Lemma 2.6, and choosing an appropriate control that depends on $x_0$ and $z$, we will estimate the size of the corresponding trajectory of (1.3), and we will establish the order of approximation of the trajectory to $z$. Later we will use these properties to compare with the trajectories of the nonlinear system (1.2). Assume that system (1.3) is controllable in $[0,\tau]$. Let $z\in\mathbb{R}^n$ and $\varepsilon>0$. We consider the control function
\[
u_\varepsilon(t)
=B^*(t)\Phi^*(t)^{-1}(\varepsilon I+M(\tau))^{-1}
\Phi(\tau)^{-1}(z-\Phi(\tau)x_0),\quad 0\leq t\leq\tau. \tag{3.1}
\]
In what follows we denote by $\beta(\tau)$ the lowest eigenvalue of $M(\tau)$. Moreover, to abbreviate the text, we denote by $C>0$ a generic constant independent of $\varepsilon>0$ and $\tau$ for $0<\tau\leq T$. We observe that $u_\varepsilon\in L^p([0,\tau],\mathbb{R}^m)$ and it follows from Lemma 2.6 that,
\[
\lVert u_\varepsilon(t)\rVert
\leq\frac{C}{\varepsilon+\beta(\tau)}\lVert B^*(t)\rVert.
\]
The corresponding trajectory of the system (1.3) is given by

\begin{align*}
x_\varepsilon(t)
&=x(t,x_0,u_\varepsilon)\\
&=\Phi(t)x_0+\int_0^t\Phi(t)\Phi(s)^{-1}B(s)B^*(s)\Phi^*(s)^{-1}
(\varepsilon I+M(\tau))^{-1}\Phi(\tau)^{-1}(z-\Phi(\tau)x_0)\,ds.
\end{align*}

In particular,

\begin{align*}
x_\varepsilon(\tau)
&=\Phi(\tau)x_0 +\int_0^\tau\Phi(\tau)\Phi(s)^{-1}B(s)B^*(s)\Phi^*(s)^{-1}
(\varepsilon I+M(\tau))^{-1}\Phi(\tau)^{-1}(z-\Phi(\tau)x_0)\,ds\\
&=\Phi(\tau)x_0+\Phi(\tau)M(\tau)(\varepsilon I+M(\tau))^{-1}
\Phi(\tau)^{-1}(z-\Phi(\tau)x_0)\\
&=\Phi(\tau)x_0+\Phi(\tau)[I-\varepsilon(\varepsilon I+M(\tau))^{-1}]
\Phi(\tau)^{-1}(z-\Phi(\tau)x_0)\\
&=z-\Phi(\tau)\varepsilon(\varepsilon I+M(\tau))^{-1}
\Phi(\tau)^{-1}(z-\Phi(\tau)x_0). 
\end{align*}

Hence 
\[ \lim_{\varepsilon\to 0} x_\varepsilon(\tau) = \lim_{\varepsilon\to 0} \Big( z-\Phi(\tau)\varepsilon(\varepsilon I+M(\tau))^{-1}
\Phi(\tau)^{-1}(z-\Phi(\tau)x_0)\Big) = z.\] 
Furthermore, using again Lemma 2.6 we obtain
\[
\lVert x_\varepsilon(\tau)-z\rVert
\leq C\frac{\varepsilon}{\varepsilon+\beta(\tau)}
\lVert z-\Phi(\tau)x_0\rVert.
\]
We can estimate the trajectory $x_\varepsilon(\cdot)$ as
\begin{align*}
    \lVert x_\varepsilon(t)-z\rVert &\leq\lVert x_\varepsilon(\tau)-z\rVert +\lVert x_\varepsilon(\tau)-x_\varepsilon(t)\rVert \\
    &\leq C\left[\frac{\varepsilon}{\varepsilon+\beta(\tau)}
+\frac{1}{\varepsilon+\beta(\tau)}\nu(\tau,t)\right]
\lVert z-\Phi(\tau)x_0\rVert+\lVert\Phi(\tau)x_0-\Phi(t)x_0\rVert,
\end{align*}
for $0\leq t\leq\tau$, where
\[
\nu(\tau,t)=\lVert\Phi(\tau)M(\tau)-\Phi(t)M(t)\rVert.
\]
which shows that for fixed $\varepsilon>0$ the trajectories $x_\varepsilon(\cdot)$ are bounded and, consequently, they behave
differently from that studied in Proposition 2.4.
When the function $\lVert A(\cdot)\rVert$ is bounded on $[0,\tau]$, there exists a constant $L>0$ such that $\lVert\Phi(t_2)-
\Phi(t_1)\rVert\leq L\lvert t_2-t_1\rvert$ for all $t_1,t_2\in[0,T]$. This allows us to reduce the preceding estimate to
\[
\lVert x_\varepsilon(t)-z\rVert
\leq C\left[\frac{\varepsilon}{\varepsilon+\beta(\tau)}
+\frac{\tau-t}{\varepsilon+\beta(\tau)}\right]
\lVert z-\Phi(\tau)x_0\rVert+\lVert\Phi(\tau)x_0-\Phi(t)x_0\rVert,
\quad 0\leq t\leq\tau. \tag{3.2}
\]
To establish our result about controllability of system (1.2), we need to introduce a certain boundedness condition. In the next statement $\beta_{T-\omega}(\omega)$ denotes the lowest eigenvalue of the matrix $M_{T-\omega}$,
which is defined in (2.7).

\textbf{(C5)} Let $C_1\geq 0$ be a constant, and let $\alpha_\omega:[0,\omega]\to[0,\infty)$, $\omega>0$, be a continuous function such
that
\[
\alpha_\omega(t)\leq C_\omega+C_1\frac{\omega}{\beta_{T-\omega}(\omega)}
+N_\Phi\int_0^t m_f(s+T-\omega)\Omega(\alpha_\omega(s))\,ds,
\quad 0\leq t\leq\omega, \tag{3.3}
\]
where $C_\omega\to 0$ as $\omega\to 0$. Then
\[
\int_0^\omega m_f(s+T-\omega)\Omega(\alpha_\omega(s))\,ds\to 0
\quad\text{as }\omega\to 0.
\]

\begin{example}
Assume that $\beta_{T-\omega}(\omega)=\omega^3$, $m_f=1$ and $\Omega(\xi)=\xi^\theta$ with $0<\theta<1/2$. The inequality (3.3) is reduced to
\[
\alpha_\omega(t)\leq C_\omega+\frac{C_1}{\omega^2}
+N_\Phi\int_0^t\alpha_\omega(s)^\theta\,ds,\quad 0\leq t\leq\omega.
\]
Applying the Bihari inequality (see \cite[Theorem 4]{16}), we can affirm that
\[
\alpha_\omega(t)^{1-\theta}
\leq\left(C_\omega+\frac{1}{\omega^2}C_1\right)^{1-\theta}
+(1-\theta)N_\Phi t,\quad 0\leq t\leq\omega. \tag{3.4}
\]
Let $\gamma=\dfrac{\theta}{1-\theta}$. Since $\theta<1/2$ it follows that $\gamma<1$, which implies that $(a+b)^\gamma\leq a^\gamma+b^\gamma$ for all $a,b>0$. Applying this property to (3.4), we obtain that
\[
\alpha_\omega(t)^\theta
\leq\left[\left(C_\omega+\frac{1}{\omega^2}C_1\right)^{1-\theta}
+(1-\theta)N_\Phi t\right]^\gamma
\leq\left(C_\omega+\frac{1}{\omega^2}C_1\right)^\theta
+(1-\theta)^\gamma N_\Phi^\gamma t^\gamma.
\]
This implies that
\[
\int_0^\omega\alpha_\omega(s)^\theta\,ds
\leq(1-\theta)^{\gamma+1}N_\Phi^\gamma\omega^{\frac{1}{1-\theta}}
+\omega C_\omega^\theta+C_1^\theta\omega^{1-2\theta}
\to 0,\quad\omega\to 0.
\]
Consequently, condition (C5) holds.
\end{example}
\noindent
We are in a position to state the main result of this paper.
\begin{theorem}
Assume that the system (1.3) is uniformly controllable on $[0,T]$, the function $\lVert A(\cdot)\rVert$
is bounded on $[0,T]$, and conditions (C1)-(C5) are fulfilled. Then system (1.2) is approximately
controllable on $[0,T]$.
\end{theorem}
\begin{proof}
Let $z\in\mathbb{R}^n$. In this proof we abbreviate our notation by writing $y(t,x_0,u)$ for the solution of (1.2)-(2.2). Let $(\omega_k)_k$, with $0<\omega_k<T$, be a nondecreasing sequence convergent to $T$. Since $y(\cdot,x_0,0)$ is a continuous function on $[0,T]$, there exists a constant $C\geq 0$ such that $\lVert y(\omega_k,x_0,0)\rVert\leq C$ for all $k\in\mathbb{N}$. We fix $\varepsilon>0$. We consider the linear system (2.5) with $\omega_k$ instead $\omega$ and initial condition $y(\omega_k,x_0,0)$. That is, we study the controllability of the system
\begin{align*}
x'(t) &= A(t+\omega_k)x(t)+B(t+\omega_k)u(t),\quad t\in[0,T-\omega_k], \tag{3.5}\\
x(0) &= y(\omega_k,x_0,0). \tag{3.6}
\end{align*}
Motivated by (3.1), we define the control function $v_\varepsilon^k:[0,T-\omega_k]\to \mathbb{R}^m$ given by
\[
v_\varepsilon^k(t)
=B^*(t+\omega_k)\Phi^*(t+\omega_k,\omega_k)^{-1}
(\varepsilon I+M_{\omega_k})^{-1}\Phi(T,\omega_k)^{-1}
\bigl(z-\Phi(T,\omega_k)y(\omega_k,x_0,0)\bigr), \tag{3.7}
\]
where $M_{\omega_k}$ is defined in (2.7) with $\omega_k$ instead of $\omega$. Let
\[
z^k=\Phi(T,\omega_k)^{-1}\bigl(z-\Phi(T,\omega_k)y(\omega_k,x_0,0)\bigr).
\]
The solution $x_\varepsilon^k(t)$ of system (3.5)-(3.6) with the control $u(t)=v_\varepsilon^k(t)$ is given by
\[
x_\varepsilon^k(t)
=\Phi(t+\omega_k,\omega_k)y_k
+\int_0^t\Phi(t+\omega_k,s+\omega_k)B(s+\omega_k)v_\varepsilon^k(s)\,ds,
\]
and substituting $v_\varepsilon^k(s)$ yields

\begin{align*}
x_\varepsilon^k(t) &=\Phi(t+\omega_k,\omega_k)y(\omega_k,x_0,0)\\
&+\int_0^t\Phi(t+\omega_k,s+\omega_k)B(s+\omega_k)B^*(s+\omega_k)
\Phi^*(s+\omega_k,\omega_k)^{-1}(\varepsilon I+M_{\omega_k})^{-1}z^k\,ds\\
=&\Phi(t+\omega_k,\omega_k)y(\omega_k,x_0,0)\\
&+\int_0^t\Phi(t+\omega_k,\omega_k)\Phi(s+\omega_k,\omega_k)^{-1}B(s+\omega_k)B^*(s+\omega_k)
\Phi^*(s+\omega_k,\omega_k)^{-1}(\varepsilon I+M_{\omega_k})^{-1}z^k\,ds\\
=&\Phi(t+\omega_k,\omega_k)y(\omega_k,x_0,0)
+\Phi(t+\omega_k,\omega_k)M_{\omega_k}(\varepsilon I+M_{\omega_k})^{-1}
\Phi(T,\omega_k)^{-1}\bigl(z-\Phi(T,\omega_k)y(\omega_k,x_0,0)\bigr)\\
=&\Phi(t+\omega_k,\omega_k)y(\omega_k,x_0,0)
+\Phi(t+\omega_k,\omega_k)\bigl[I-\varepsilon(\varepsilon I+M_{\omega_k})^{-1}\bigr]
\Phi(T,\omega_k)^{-1}\bigl(z-\Phi(T,\omega_k)y(\omega_k,x_0,0)\bigr)\\
=&\Phi(t+\omega_k,\omega_k)\Phi(T,\omega_k)^{-1}z -\Phi(t+\omega_k,\omega_k)\varepsilon(\varepsilon I+M_{\omega_k})^{-1}
\Phi(T,\omega_k)^{-1}\bigl(z-\Phi(T,\omega_k)y(\omega_k,x_0,0)\bigr).
\end{align*}

From the last term it follows that

\begin{align*}
x_\varepsilon^k(s-\omega_k)-z
=&\Phi(s,\omega_k)\Phi(T,\omega_k)^{-1}z-\Phi(s,\omega_k)^{-1}z\\
&-\Phi(s,\omega_k)\varepsilon(\varepsilon I+M_{\omega_k})^{-1}
\Phi(T,\omega_k)^{-1}\bigl(z-\Phi(T,\omega_k)y(\omega_k,x_0,0)\bigr)
\end{align*}

for $\omega_k\leq s\leq T$. In particular,
\[
x_\varepsilon^k(T-\omega_k)
=z-\Phi(T,\omega_k)\varepsilon(\varepsilon I+M_{\omega_k})^{-1}
\Phi(T,\omega_k)^{-1}\bigl(z-\Phi(T,\omega_k)y(\omega_k,x_0,0)\bigr). \tag{3.9}
\]
Moreover, using (3.2), we can assert that
\[
\lVert x_\varepsilon(s)-z\rVert
\leq C\left[
\frac{\varepsilon}{\varepsilon+\beta_{\omega_k}(T-\omega_k)}
+\frac{T-s}{\varepsilon+\beta_{\omega_k}(T-\omega_k)}
+(T-s)\right],\quad \omega_k<s<T, \tag{3.10}
\]
where $C$ is a constant independent of $\varepsilon$ and $k$, and $\beta_{\omega_k}(T-\omega_k)$ is the lowest eigenvalue of $M_{\omega_k}$. We define the function
\[
u_\varepsilon^k(s)=
\begin{cases}
0, & 0\leq s<\omega_k,\\
v_\varepsilon^k(s-\omega_k), & \omega_k\leq s\leq T.
\end{cases}
\]
We have
\[
y(t,x_0,u_\varepsilon^k)=\Phi(t,0)x_0+\int_0^t\Phi(t,s)f(s,y(s,x_0,u_\varepsilon^k))\,ds
+\int_0^t\Phi(t,s)B(s)u_\varepsilon^k(s)\,ds,
\]
for $0\leq t\leq T$. Using the uniqueness of solutions of (1.2)-(2.2), we obtain that $y(t,x_0,u_\varepsilon^k)=y(t,x_0,0)$ for $0\leq t\leq\omega_k$. Hence follows that
\[
y(\omega_k,x_0,0)=\Phi(\omega_k,0)x_0+\int_0^{\omega_k}\Phi(\omega_k,s)f(s,y(s,x_0,u_\varepsilon^k))\,ds,\quad k\in\mathbb{N}. \tag{3.11}
\]
\newpage
\noindent
In addition, for $\omega_k\leq t\leq T$, we have
\begin{align*}
y(t,x_0,u_\varepsilon^k) &=\Phi(t,0)x_0+\int_0^{\omega_k}\Phi(t,s)f(s,y(s,x_0,u_\varepsilon^k))\,ds
+\int_{\omega_k}^t\Phi(t,s)f(s,y(s,x_0,u_\varepsilon^k))\,ds\\
&+\int_{\omega_k}^t\Phi(t,s)B(s)u_\varepsilon^k(s)\,ds\\
=&\Phi(t,0)x_0+\Phi(t,\omega_k)\int_0^{\omega_k}\Phi(\omega_k,s)
f(s,y(s,x_0,u_\varepsilon^k))\,ds +\int_{\omega_k}^t\Phi(t,s)f(s,y(s,x_0,u_\varepsilon^k))\,ds\\
&+\int_0^{t-\omega_k}\Phi(t,\xi+\omega_k)B(\xi+\omega_k)v_\varepsilon^k(\xi)\,d\xi.
\end{align*}
Combining the last expression with (3.11), we obtain
\[
y(t,x_0,u_\varepsilon^k)=\Phi(t,\omega_k)y_k
+\int_{\omega_k}^t\Phi(t,s)f(s,y(s,x_0,u_\varepsilon^k))\,ds
+\int_0^{t-\omega_k}\Phi(t,\xi+\omega_k)B(\xi+\omega_k)v_\varepsilon^k(\xi)\,d\xi,
\]
for $\omega_k\leq t\leq T$. Replacing $x_\varepsilon^k(T-\omega_k)$ from (3.9), this yields that

\begin{align*}
y(T,x_0,u_\varepsilon^k)
=&\Phi(T,\omega_k)y(\omega_k,x_0,0)
+\int_{\omega_k}^T\Phi(T,s)f(s,y(s,x_0,u_\varepsilon^k))\,ds\\
&+\int_0^{T-\omega_k}\Phi(T,\xi+\omega_k)B(\xi+\omega_k)v_\varepsilon^k(\xi)\,d\xi\\
=&x_\varepsilon^k(T-\omega_k)
+\int_{\omega_k}^T\Phi(T,s)f(s,y(s,x_0,u_\varepsilon^k))\,ds.
\end{align*}
We now compare $y_\varepsilon^k(t)$ with the solution of the linear system

\begin{align*}
\psi'(t)&=A(t)\psi(t)+B(t)u_\varepsilon^k(t),\quad \omega_k\leq t\leq T,\\
\psi(\omega_k)&=y(\omega_k,x_0,0).
\end{align*}

Changing the variable $t=\tau+\omega_k$, we see that $\psi(t)=x_\varepsilon^k(t-\omega_k)$ for $\omega_k\leq t\leq T$. This allows us to
write
\[
y(t,x_0,u_\varepsilon^k)=\psi(t)+\int_{\omega_k}^t\Phi(t,s)f(s,y(s,x_0,u_\varepsilon^k))\,ds
\]
and, consequently,
\[
y(t,x_0,u_\varepsilon^k)-z=\psi(t)-z+\int_{\omega_k}^t\Phi(t,s)f(s,y(s,x_0,u_\varepsilon^k))\,ds. \tag{3.12}
\]
which implies
\[
\lVert y(t,x_0,u_\varepsilon^k)\rVert
\leq \lVert\psi(t)-z\rVert+\lVert z\rVert
+N_\Phi\int_{\omega_k}^t m_f(s)\Omega(\lVert y(s,x_0,u_\varepsilon^k)\rVert)\,ds
\]
for $\omega_k\leq t\leq T$. It follows from (3.10) that
\begin{align*}
\lVert y(\tau+\omega_k,x_0,u_\varepsilon^k)\rVert
\leq& C\left[
\frac{\varepsilon}{\varepsilon+\beta_{\omega_k}(T-\omega_k)}
+\frac{T-\omega_k-\tau}{\varepsilon+\beta_{\omega_k}(T-\omega_k)}
+(T-\omega_k-\tau)\right]+\lVert z\rVert\\
&+N_\Phi\int_0^\tau m_f(s+\omega_k)
\Omega(\lVert y(s+\omega_k,x_0,u_\varepsilon^k)\rVert)\,ds \tag{3.13}
\end{align*}
for $0\leq\tau\leq T-\omega_k$. As $\varepsilon>0$ is independent of $k$, we can take $\varepsilon=\gamma_k=T-\omega_k$. Substituting in
(3.13), we can write
\[
\lVert y(\tau+\omega_k,x_0,u_\varepsilon^k)\rVert
\leq C[2+(T-\omega_k-\tau)]+\lVert z\rVert
+N_\Phi\int_0^\tau m_f(s+\omega_k)
\Omega(\lVert y(s+\omega_k,x_0,u_\varepsilon^k)\rVert)\,ds,
\]
for $0\leq\tau\leq T-\omega_k$. Applying Lemma 2.11 we derive that
$\lVert y(\tau+\omega_k,x_0,u_\varepsilon^k)\rVert\leq N_1$ for $0\leq\tau\leq T-\omega_k$,
$k\in\mathbb{N}$, and for some constant $N_1>0$. On the other hand, since $\Omega(\cdot)$ is an absolutely continuous function,
\begin{align*}
\Omega(\lVert y(s+\omega_k,x_0,u_\varepsilon^k)\rVert)
-\Omega(\lVert y(s+\omega_k,x_0,u_\varepsilon^k)-z\rVert)
&=\int_{\lVert y(s+\omega_k,x_0,u_\varepsilon^k)-z\rVert}^{\lVert y(s+\omega_k,x_0,u_\varepsilon^k)\rVert}
\Omega'(\xi)\,d\xi\\
&\leq\int_{\lVert y(s+\omega_k,x_0,u_\varepsilon^k)-z\rVert}^{N_1}
\Omega'(\xi)\,d\xi\\
&\leq\Omega(N_1).
\end{align*}
We define
\[
\alpha_k(\tau)=\lVert y(\tau+\omega_k,x_0,u_\varepsilon^k)-z\rVert,
\quad 0\leq\tau\leq T-\omega_k. \tag{3.14}
\]
It follows from (3.10) and (3.12) that
\begin{align*}
\alpha_k(\tau)
\leq& C\left[
\frac{\varepsilon}{\varepsilon+\beta_{\omega_k}(T-\omega_k)}
+\frac{T-\omega_k-\tau}{\varepsilon+\beta_{\omega_k}(T-\omega_k)}
+(T-\omega_k-\tau)\right] +N_\Phi\int_0^\tau m_f(s+\omega_k)\,ds\,\Omega(N_1)\\
&+N_\Phi\int_0^\tau m_f(s+\omega_k)\Omega(\alpha_k(s))\,ds \tag{3.15}
\end{align*}
for $0\leq\tau\leq T-\omega_k$. Using again $\varepsilon=\gamma_k$ and substituting in (3.15), we can write
\[
\alpha_k(\tau)
\leq C_k+C\left[
2\frac{\gamma_k}{\beta_{\omega_k}(T-\omega_k)}
+(T-\omega_k-\tau)\right]
+N_\Phi\int_0^\tau m_f(s+\omega_k)\Omega(\alpha_k(s))\,ds, \tag{3.16}
\]
for $0\leq\tau\leq T-\omega_k$ and where $C_k\to 0$ as $k\to\infty$. Combining with condition (C5), we infer that
\[
\int_0^{T-\omega_k}m_f(s+\omega_k)\Omega(\alpha_k(s))\,ds\to 0
\quad\text{as }k\to\infty.
\]
Using now (3.16) with $\tau=T-\omega_k$, we obtain that
$\alpha_k(T-\omega_k)\to 0$ as $k\to\infty$. Substituting in
(3.14), we have that
\[
\lVert y(T,x_0,u_\varepsilon^k)-z\rVert
=\alpha_k(T-\omega_k)\to 0,\quad k\to\infty.
\]
Hence, $z\in\overline{R(T,x_0,f)}$, which shows that $R(T,x_0,f)$ is dense in $\mathbb{R}^n$. This completes the proof.
\end{proof}

\begin{example}
    We complete the text with an application. We apply our results to study the controllability of the system
    \begin{align*}
        x'(t)&=Ax(t)+f(x(t))+Bu(t),\quad 0\leq t\leq T, \tag{3.17}\\
        x(0)&=x_0\in\mathbb{R}^2, \tag{3.18}
    \end{align*}
where $A=\begin{pmatrix}0&1\\0&0\end{pmatrix}$, $B=\begin{pmatrix}0\\1\end{pmatrix}$ and $f:\mathbb{R}^2\to\mathbb{R}^2$ is a locally Lipschitz continuous function. It follows from the Kalman criterion that the linear system
\[
x'(t)=Ax(t)+Bu(t),\quad 0\leq t\leq T,
\]
is controllable. We will assume that the function $f$ satisfies the H\"older condition
\[
\lVert f(x)\rVert\leq\lVert x\rVert^\theta,\quad x\in\mathbb{R}^2,\quad 0<\theta<1/2.
\]
This implies that $f$ satisfies conditions (C1)-(C4) with $m_f(s)=1$ and $\Omega(\xi)=\xi^\theta$.
Applying our results in Section 2, we obtain that (3.17)-(3.18) has a unique solution. On the other hand, since $\Phi(t,0)=e^{tA}$, we have that $\Phi(s,\omega)=e^{(s-\omega)A}$ and, it follows from (2.6) that

\begin{align*}
M_\omega(\tau)
&=\int_\omega^{\omega+\tau}
\begin{pmatrix}
1&-(s-\omega)\\
0&1
\end{pmatrix}
\begin{pmatrix}
0\\
1
\end{pmatrix}
\begin{pmatrix}
0&1
\end{pmatrix}
\begin{pmatrix}
1&0\\
-(s-\omega)&1
\end{pmatrix}\,ds\\
&=
\begin{pmatrix}
\tau^3/3&-\tau^2/2\\
-\tau^2/2&\tau
\end{pmatrix}.
\end{align*}

By calculating the eigenvalues of $M_\omega(\tau)$, we obtain that
$\dfrac{\tau^2}{12}\dfrac{\tau}{\beta_{T-\tau}(\tau)}\leq 1$ for $0<\tau\leq\tau_0$ and some
$\tau_0>0$. The inequality (3.3) is reduced to
\[
\alpha_\tau(t)\leq C_\tau+C_1\frac{\tau}{\beta_{T-\tau}(\tau)}
+N\int_0^t\alpha_\tau(s)^\theta\,ds,\quad 0\leq t\leq\tau,
\]
where $N=\sup_{0\leq t\leq T}\lVert e^{tA}\rVert$, $C_1>0$ is a constant and $C_\tau\to 0$ as $\tau\to 0$. Replacing the previous
estimate for $\beta_{T-\tau}(\tau)$, we have
\[
\alpha_\tau(t)\leq C_\tau+C_1\frac{12}{\tau^2}
+N\int_0^t\alpha_\tau(s)^\theta\,ds,\quad 0\leq t\leq\tau,
\]
Proceeding as in Example 3.2 we can assert that condition (C5) holds. Theorem 3.3 allows us to
conclude that the system (3.17)-(3.18) is approximately controllable in the interval $[0,T]$.
\end{example}
\noindent

\end{document}